\documentclass{compositio}

\usepackage{here}
\usepackage{amsmath, amsthm, amssymb, amsfonts,mathrsfs,ascmac,amscd} 
\usepackage%[dvipdfmx]
{xcolor}
\usepackage{enumitem}
\usepackage{mathtools}
\usepackage{comment, url}
\usepackage{multirow}
\usepackage[all,knot,poly]{xy}
\usepackage{tikz} 
\usepackage{bm}
\usepackage%[nameinlink]
{cleveref}

\usepackage{longtable}
\usepackage[format=plain, font=normalsize %,labelfont=bf
]{caption} 

\numberwithin{equation}{section}

\newtheorem{thm}{Theorem}[section]
\newtheorem{lem}[thm]{Lemma}

\newtheorem{prop}[thm]{Proposition}

\theoremstyle{remark}
\newtheorem{ack}{Acknowledgments\!}

\theoremstyle{definition}

\newtheorem{eg}[thm]{Example}  
\newtheorem{rem}[thm]{Remark}

\newtheorem{q}[thm]{Question} 

\newtheorem{def/prop}[thm]{Definition/Proposition}

\newcommand{\pmx}[1]{\begin{pmatrix}#1\end{pmatrix}}
\newcommand{\spmx}[1]{{\small \pmx{#1}}}

\usepackage[pagewise, mathlines]{lineno} %\linenumbers 
\usepackage{time}

\numberwithin{equation}{section}
\usepackage{etoolbox}          %% <- for \cspreto, \csappto and \patchcmd

\newcommand*\linenomathpatch[1]{%
  \cspreto{#1}{\linenomath}%
  \cspreto{#1*}{\linenomath}%
  \csappto{end#1}{\endlinenomath}%
  \csappto{end#1*}{\endlinenomath}%
}
\newcommand*\linenomathpatchAMS[1]{%
  \cspreto{#1}{\linenomathAMS}%
  \cspreto{#1*}{\linenomathAMS}%
  \csappto{end#1}{\endlinenomath}%
  \csappto{end#1*}{\endlinenomath}%
}
\expandafter\ifx\linenomath\linenomathWithnumbers 
 \let\linenomathAMS\linenomathWithnumbers
\patchcmd\linenomathAMS{\advance\postdisplaypenalty\linenopenalty}{}{}{}
\else
  \let\linenomathAMS\linenomathNonumbers
\fi

\linenomathpatch{equation}
\linenomathpatchAMS{gather}
\linenomathpatchAMS{multline}
\linenomathpatchAMS{align}
\linenomathpatchAMS{alignat}
\linenomathpatchAMS{flalign}
\makeatletter
\patchcmd{\mmeasure@}{\measuring@true}{
  \measuring@true
  \ifnum-\linenopenaltypar>\interdisplaylinepenalty
    \advance\interdisplaylinepenalty-\linenopenalty
  \fi
  }{}{}
\makeatother

\newcommand{\bb}[1]{{\mathbb{#1}}}

\newcommand{\bs}[1]{{\boldsymbol{#1}}}

\newcommand{\N}{\bb{N}}
\newcommand{\Z}{\bb{Z}}
\newcommand{\Zp}{\bb{Z}_{p}}
\newcommand{\Q}{\bb{Q}}

\newcommand{\R}{\bb{R}}
\newcommand{\C}{\bb{C}}

\newcommand{\F}{\bb{F}}

\newcommand{\ccdot}{\!\cdot\!}

\newcommand{\ol}{\overline}

\makeatletter
\@namedef{subjclassname@2020}{
\textup{2020} Mathematics Subject Classification}
\makeatother 

\subjclass[2020]{Primary 20C12, 57M12; Secondary 57K10, 20E26} 

\keywords{knot, Alexander polynomial, liminal representation, Lucas sequence, arithmetic topology.} 
\title[Liminal ${\rm SL}_2\Zp$-representations]{%{\large 
{\Large Liminal ${\rm SL}_2\Zp$-representations and odd-th cyclic covers\\ of 
genus one two-bridge knots}}%}

\author{%\large 
Honami Sakamoto} %$^{\ast 1}$} 
\email{sakamo10ho73@gmail.com} 
\address{%$\,^{\ast 1}$ 
Department of Mathematics, Faculty of Science, Ochanomizu University; 2-1-1 Otsuka, Bunkyo-ku, 112-8610, Tokyo, Japan}

\author{%\large 
Ryoto Tange%$^{\ast 2}$
} 
\email{rtange.math@gmail.com}
\address{%$\,^{\ast 2}$ 
Center for Promotion of Higher Education, 
Kogakuin University; 2665-1 Nakano, Hachioji, 192-0015, Tokyo, Japan}
\author{%\large 
Jun Ueki%$^{\ast 3}$
} 
\email{uekijun46@gmail.com}
\address{%$\,^{\ast 3}$ 
Department of Mathematics, Faculty of Science, Ochanomizu University; 2-1-1 Otsuka, Bunkyo-ku, 112-8610, Tokyo, Japan}

\begin{document}

\begin{abstract} 
Let $K$ be a genus one two-bridge knot. Let $p$ be a prime number and let $\Zp$ denote the ring of $p$-adic integers. 
In the spirit of arithmetic topology, we observe that if $p\neq 2$ and $p$ divides (or $p=2$ and $2^3$ divides) the size of the 1st homology group of some odd-th cyclic branched cover of the knot $K$, then its group $\pi_1(S^3-K)$ admits a liminal ${\rm SL}_2\Z_p$-character. 
In addition, we discuss the existence of liminal ${\rm SL}_2\Zp$-representations and give a remark on a general two-bridge knot. 
In the course of argument, we also point out a constraint for prime numbers dividing certain Lucas-type sequences by using the Legendre symbols. 
\end{abstract}

\maketitle 

{\small 
\tableofcontents
} 
\section{Introduction} 
Let $p$ be a prime number and let $\Zp$ denote the ring of $p$-adic integers. 
The analogy between knots and primes, or 3-manifolds and the rings of integers of number fields, has played an important role since Gauss's era. In modern times, Barry Mazur initially pointed out the analogy between Iwasawa theory of $\Zp$-extensions and Alexander--Fox theory of $\Z$-covers \cite{Mazur1963}, and such analogies have been systematically arranged by Kapranov \cite{Kapranov1995}, Reznikov \cite{Reznikov1997, Reznikov2000Selecta}, Morishita \cite{Morishita2002, Morishita2012, Morishita2024}, M.~Kim \cite{MKim2020}, and others. 
The theories of $\Zp$-extensions and $\Z$-covers may be seen as deformation theories of ${\rm GL}_1$-representations, and this viewpoint extends to the analogy between Hida--Mazur's Galois deformation theory and Thurston's hyperbolic deformation theory \cite{MorishitaTerashima2007}.  

In the context of Hida--Mazur theory and explorations of its more precise analogue in low dimensional topology 
\cite{MTTU2017, KMTT2018, TangeTranUeki2022IMRN, BenardTangeTranUeki-Whitehead, KMTT2023, TangeR2025ProcLDTNT, TangeUeki2024MathNach}, 
there are special interests in irreducible ${\rm SL}_2\Zp$-representations whose residual ${\rm SL}_2\F_p$-representations are reducible. 
%Such a representation turns out to be {\it liminal} in the sense of Mazur, who suggested to ``go the other way'' \cite[Section 19]{Mazur2011BAMS}. 
%
In this paper, following Mazur's suggestion \cite[Section 19]{Mazur2011BAMS}, 
we aim to ``go the other way'', by considering {\it liminal} representations and characters. %study in the opposite direction. 

Let $\pi$ be a group. 
A function $\chi:\pi \to \Zp$ is called \emph{an ${\rm SL}_2\Zp$-character} if 
there exists an ${\rm SL}_2$-representation $\rho$ over an extension of $\Zp$ such that $\chi={\rm tr}\,\rho$ holds.
An ${\rm SL}_2\Zp$-representation (resp. character) is said to be {\it liminal} if it is reducible and its every open neighborhood contains an irreducible ${\rm SL}_2\Zp$-representation (resp. character).  
Our main result in this paper will be the following observation.   

\begin{thm} \label{thm} 
Let $K$ be a genus one two-bridge knot in $S^3$. 
If an odd prime number $p$ (resp.~$2^3$) divides the size of the 1st homology group of some odd-th cyclic branched cover of $K$, then 
its group $\pi_1(S^3-K)$ admits a liminal ${\rm SL}_2\Zp$-character (resp.~${\rm SL}_2\Z_2$-character).
\end{thm} 

One may find this result interesting, partially for the following reason. 
It has been classically known by Burde and de Rham \cite{Burde1967MathAnn, deRham1967EnseignMath} that such deformative characters correspond to the zeros of the Alexander polynomial 
(see also \cite{HeusenerPortiSuarezPeiro2001Crelle}), 
and Fox--Weber's formula asserts that the size of $H_1$ of a finite cyclic cover is calculated by using the Alexander polynomial,  
so our result has a deep context in its background. 
However, there is no obvious implication, and this result would rather suggest a new interaction between ${\rm GL}_1$-deformation theory and ${\rm SL}_2$-deformation theory. 
We will describe further queries and perspectives 
for readers with specific interests in \Cref{rem.final}.  
We also remark that a $p$-adic analogue of Burde--de Rham theorem and its application to number theoretical settings are due to the second author and others \cite{MizusawaTangeTerashima2025IMRN}.

To prove \Cref{thm}, we utilize the character varieties, the cyclic resultants of the Alexander polynomials, the Legendre symbols, Hensel's lemma, and establish a property of certain Fibonacci/Lucas-type sequences as well. 

The argument will be given as follows. 
In Section 2, we recollect properties of genus one two-bridge knots, namely, double twist knots of type $J(2k,2l)$ with $(0,0)\neq (k,l)\in \Z^2$ and calculate the intersection of the varieties of irreducible and reducible ${\rm SL}_2\C$-characters (\Cref{prop.Jredirr}). 
In Section 3, by using Hensel's lemma several times, 
we prove that $J(2k,2l)$ admits a liminal ${\rm SL}_2\Zp$-character iff 
(i) $p=2$ and $r\equiv 1$ mod 8 or 
(ii) $p\neq 2$ and the Legendre symbol satisfies $(\frac{r}{p})=1$, 
where $r$ denotes the square-free part of $4k^2l^2-kl$  (\Cref{thm.limchar}), and calculate several examples. 
In Section 4, we further discuss the existence of liminal ${\rm SL}_2\Zp$-representations (\Cref{prop.limZp}), 
and give a remark on the case of a general two-bridge knot (\Cref{rem.2bridge}).

In Section 5, for a given $m\in \Z$, we write $t^2-t+m=(t-a)(t-b)$ and define the Lucas-type sequence $(L_n)_n\in \Z^\N$ by $L_n=a^n+b^n$. 
We prove that the condition $p\mid L_{2n+1}$ for some $n\in \Z_{\geq0}$ implies $(\frac{4m^2-m}{p})=1$ (\Cref{thm.Lucas}) 
and exhibit examples. 
This part is tailored towards readers with independent interests also. 
Note that the classical one with $m=-1$ appears in the case of the figure-eight knot $J(2,-2)=4_1$. 

In Section 6, we recall some properties of the cyclic covers $M_n\to S^3$ of a knot $K$, 
prove for a knot $K$ with its Alexander polynomial $\Delta_K(t)=mt^2-(2m-1)t+m$ that $\#H_1(M_{2n+1})=L_{2n+1}^{\,2}$ holds, 
yielding that $p\mid \#H_1(M_{2n+1})$ implies $(\frac{4m^2-m}{p})=1$ (\Cref{thm.pmidrn}), 
and deduce \Cref{thm}. We also attach further remarks (\Cref{rem.final}).

\begin{ack} We are grateful to L\'eo B\'enard, Tomoki Mihara, Tatsuya Ohshita, Shin-ichiro Seki, Adam Sikora, Motoo Tange, Yuji Terashima, Anh T.~Tran, Yoshikazu Yamaguchi, and Hyuga Yoshizaki for their useful comments. 
Numerical calculations of examples in this article were partially enhanced by ChatGPT and Grok.  
The third author has been partially supported by JSPS KAKENHI Grant Number JP23K12969. 
\end{ack} 

\section{Genus one two-bridge knots} %Double twist knots of type $J(2k,2l)$} 
\label{sec.DTK} 
It is known that every genus one two-bridge knot is realized as a \emph{double twist knot} of type $J(2k,2l)$ with $(0,0)\neq (k,l)\in \Z^2$ 
defined by the following diagram. 
A basic reference is \cite{Tran2018Kodai}. 

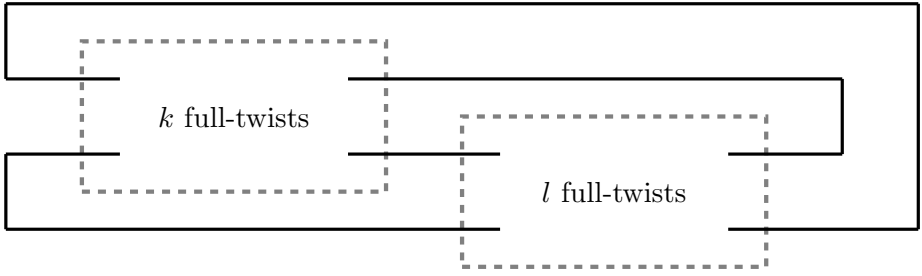
\begin{center}
\begin{tikzpicture}
\draw[ultra thick, gray, dashed] (-5,1.5) rectangle (-1,-0.5); 
\node at (-3,0.5) {$k$ full-twists};
\draw[ultra thick, gray, dashed] (0,0.5) rectangle (4,-1.5); 
\node at (2,-0.5) {$l$ full-twists};

\draw[very thick] (-6,2) -- (6,2);
\draw[very thick] (-6,2) -- (-6,1);
\draw[very thick] (-6,1) -- (-4.5,1);
\draw[very thick] (-1.5,1) -- (5,1);
\draw[very thick] (5,1) -- (5,0); 
\draw[very thick] (5,0) -- (3.5,0);
\draw[very thick] (0.5,0) -- (-1.5,0);
\draw[very thick] (-4.5,0) -- (-6,0);
\draw[very thick] (-6,0) -- (-6,-1);
\draw[very thick] (-6,-1) -- (0.5,-1);
\draw[very thick] (3.5,-1) -- (6,-1);
\draw[very thick] (6,-1) -- (6,2);
\end{tikzpicture} 
\end{center} 

Especially, $J(2,2l)$ is known as a twist knot; $J(2,0)$ is the unknot $0_1$, $J(2,2)$ is the trefoil $3_1$, and $J(2,-2)$ is the figure-eight knot $4_1$. 

The group of such a knot admits a presentation \[\pi:=\pi_1(S^3-J(2k,2l))=\langle a,b\mid w^la=bw^l\rangle\] with $a,b$ being meridians and $w=(ba^{-1})^k(b^{-1}a)^k$. 

A Seifert matrix is given by $V=\spmx{k&1\\0&l}$, so its Alexander polynomial becomes 
\[\Delta_{J(2k,2l)}(t)={\rm det}\,(tV-V^{\perp})=klt^2+(1-2kl)t+kl.\] 

For each $g\in \pi$, 
%we associate the map ${\rm tr}\, g: {\rm Hom}(\pi,{\rm SL}_2\C)\to \C; \rho\mapsto {\rm tr}\,\rho(g)$.
let ${\rm tr}\, g$ denote the map ${\rm Hom}(\pi,{\rm SL}_2\C)\to \C; \rho\mapsto {\rm tr}\,\rho(g)$. 
Then the conjugacy classes of ${\rm SL}_2\C$-representations are parametrized by 
\[x:={\rm tr}\,a, \ \ y:={\rm tr}\,ab^{-1}.\]  
(This $y$ is different from that in \cite{TangeUeki2024MathNach} and has a strong advantage in our calculation. 
The old $y$ coincides with 
${\rm tr}\,ab={\rm tr}\,a\,{\rm tr}\,b-{\rm tr}\,ab^{-1}=x^2-y$.) 
We further put $z:={\rm tr}\,w$. 
% y -> x^2-y 
% x^2-y-2 -> y-2
% 2-y -> 2-(x^2-y)=-x^2+y+2

Each conjugacy class of non-abelian ${\rm SL}_2\C$-representation is presented by \emph{Riley's universal representation} 
$\rho^{\rm R}:\pi\to {\rm SL}_2\C$ defined by
\[\rho^{\rm R}(a)=\spmx{s&1\\0&s^{-1}}, \ \ \rho^{\rm R}(b)=\spmx{s&0\\ 2-y&s^{-1}},\] 
where $x=s+s^{-1}$ and $x,y$ satisfy Riley's polynomial equation $f_{k,l}(x,y)=\Phi_{k,l}(x,y-2)=0$ 
with $f_{k,l}(x,y)\in \Z[x,y]$ (See also \cite[Theorem 3.3.1]{Le1993}, \cite{Riley1984, Riley1985}).  

There is a bijective correspondence between conjugacy classes of irreducible representations and points on $f_{k,l}(x,y)=0$ with $y-2\neq 0$. 
Each point of $y-2=0$ corresponds to a pair of the class of reducible representations and that of abelian representations. 
%We will prove 
We claim that  

\begin{prop} \label{prop.Jredirr}
The intersection of $f_{k,l}(x,y)=0$ and $y-2=0$ is $(x,y)=(\pm\sqrt{4-\frac{1}{kl}}, 2)$. 
\end{prop}

For each $n\in \Z$, let $S_n(z)\in \Z[z]$ denote the $n$-th Chebyshev polynomial of the second type defined by $S_{n-1}(2\cos\theta)=\frac{\sin n\theta}{\sin\theta}$ with $\theta\in \R$, $\sin \theta\neq 0$.  
This is equivalent to say that $S_{-1}(z)=0$, $S_0(z)=1$, and $S_{n+1}(z)-zS_n(z)+S_{n-1}(z)=0$ for every $n\in \Z$. 
These polynomials satisfy $S_{-1-n}(z)=-S_{-1+n}(z)$ and $S_{n-1}(\pm 2)=(\pm 1)^{n-1}n$. 

By using the ${\rm SL}_2$ trace relations, we may obtain 
%We have 
%We further put %$z:={\rm tr}\,w=2+(y-2)(-x^2+y+2)S_{m-1}^2(y).$ 
%In addition, we put
\[z={\rm tr}\,w=2+(y-2)(-x^2+y+2)S_{m-1}^2(y).\]

 %and $\eta:={\rm tr}\, ab^{-1}$, so that we have 
%\begin{gather*}
%\eta={\rm tr}\, ab^{-1}={\rm tr}\, a\, {\rm tr}\, b-{\rm tr}\, ab =x^2-y,\\
%z={\rm tr}\,w=2+(y-2)(y+2-x^2)S_{m-1}^2(y)\\
%\hspace{2cm} =2+(x^2-\eta-2)(-\eta+2)S_{m-1}^2(y).
%\end{gather*}
%
By \cite[Subsection 2.2]{Tran2018Kodai}, 
the variety of irreducible characters is given by 
\[f_{k,l}(x,y)=%\Phi_{k,l}(s,Y)=
S_l(z)-(1+(-x^2+y+2)S_{k-1}(y)(S_k(y)-S_{k-1}(y))S_{l-1}(z).\]
%Then \cite[Subsection 2.2]{Tran2018Kodai} asserts that 
%\begin{align*} %\[
%f_{k,l}(x,y)&=%\Phi_{k,l}(s,Y)=
%S_l(z)-(1+(-x^2+y+2)S_{k-1}(y)(S_k(y)-S_{k-1}(y))S_{l-1}(z)\\
%&=
%S_l(z)-(1+(2-\eta)S_{k-1}(y)(S_k(y)-S_{k-1}(y))S_{l-1}(z). 
%\end{align*}%\] 
%where $\eta-s^2-s^{-2}=(x^2-y)-(x^2-2)=2-y$. 

\begin{proof} [Proof of \Cref{prop.Jredirr}] 
If $y-2=0$, then we have $z=2$, 
$0=f_{k,l}(x,2)=l-(1+(-x^2+2+2)k((k+1)-k))l=1-(4-x^2)kl$, 
and hence $x=\pm\sqrt{4-\frac{1}{kl}}$. 
%If $x^2-y-2=0$, then we have $z=2$, $\eta=2$, and hence 
%$0=f_{k,l}(x,y)%=S_l(z)-(1+(Y-s^2-s^{-2})S_{k-1}(Y)(S_k(Y)-S_{k-1}(Y))S_{l-1}(z) 
%=l+1-(1+(2-y)k((k+1)-k))l
%=l+1-(1+(2-y)k)l
%=1-(2-y)kl$. This yields the assertion. 
\end{proof} 

\begin{rem} 
In the old parameters, the intersection points become $(x,\eta)=(\pm\sqrt{4-\frac{1}{kl}}, -\frac{1}{kl})$. 
Note that 
\Cref{prop.Jredirr} persists over any algebraically closed field $\bb{F}$ with ${\rm char}\,\F\nmid kl$. 
%as stated for $J(2,2l)$ in \cite[Proposition 9.1]{TangeUeki2024MathNach}. 
\end{rem}

%\begin{lem} \label{lem.F}
%Regard $F(x,\eta)=f_{k,l}(x,y)$. 
%Then $F(x,2)$ has a single root in $\Zp$ iff ? $p\neq 2$ and $p\mid 1-4kl$. 
%Then $F(x,2)$ mod $p$ has a single root in $\ol{\F}_p$ iff $p\neq 2$ and $p\mid 1-4kl$. 
%\end{lem} 

%\begin{proof} By $\eta=2$, we have $z=2$, and $F(x,2)=(l+1)-(1+(4-x^2)k)l=1-(4-x^2)kl$.  
%If $p\mid kl$, then $F(x,2)$ mod $p$ has no root. 
%If $p\neq 2$, So, $F(x,2)$ and $dF(x,2)/dx=2klx$ have a common root iff $p=2$ or $F(0,2)=1-4kl\equiv 0$ mod $p$. 
%\end{proof}

\section{Liminal ${\rm SL}_2\Zp$-characters} 
\begin{thm} \label{thm.limchar} 
Suppose $kl\neq 0$. 
Then the group of $J(2k,2l)$ admits a liminal ${\rm SL}_2\Zp$-character iff 

{\rm i)} $p=2$ and $4k^2l^2-kl \equiv 1\ {\rm mod}\ 8$ or 

{\rm ii)} $p\neq 2$, $p\nmid kl$,  and the Legendre symbol satisfies $\left(\dfrac{r}{p}\right)=1$, 
where $r$ denotes the square-free part of $4k^2l^2-kl$. 
\end{thm} 

Here, we define \emph{the square-free part} of $0\neq a\in \Z$ by $a/b^2$, where $b$ denotes the maximal integer with $b^2\mid a$. (Note that some people use this term for other notions.) 
\emph{The Legendre symbol} $(\frac{\,a\,}{p})\in \{0,\pm 1\}$ of $a\in \Z$ over $p$ is defined as follows: 
If $p\mid a$, we put $(\frac{\,a\,}{p})=0$. Suppose $p\nmid a$. 
If $a\equiv x^2$ mod $p$ for some $x\in \Z$, then we put $(\frac{\,a\,}{p})=1$. 
Otherwise, we put $(\frac{\,a\,}{p})=-1$.
The following lemmas are %elementary 
consequences of Hensel's lemma (cf.\cite[Chapter II (4.6)]{Neukirch}). 

\begin{lem} \label{lem.Hensel} 
Let %$0\neq 
$a \in \Z$ with $p^2\nmid a$. Then,  we have $\sqrt{a}\in \Zp$ iff 

{\rm i)} $p=2$ and $a\equiv 1\ {\rm mod}\ 8$ or 

{\rm ii)} $p\neq 2$ and $(\frac{\,a\,}{p})=1$.    
\end{lem}

\begin{proof} 
(i) 
Suppose $p=2$. If $\sqrt{a}\in \Z_2$, then $\sqrt{a}=\sum_i b_i 2^i$ with $b_i\in \{0,1\}$, and $a\equiv (b_0+2b_1+4b_2)^2 \equiv b_0^2+4b_1(b_0+b_1)$ mod 8. 
Hence, by the assumption $4\nmid a$, we have $b_0=1$. 
In both cases $b_1=0,1$, we have that $a\equiv 1$ mod $8$. 
Conversely, suppose that $a\equiv 1$ mod 8 and write $a=8b+1$ with $b\in \Z$. 
If $X^2-a=0$ have a solution $X=\alpha$ in $\Z_2$, then $\alpha=2\beta+1$ for some $\beta\in \Z_2$,
since otherwise we have $\alpha=2\beta$ for some $\beta\in \Z_2$ and $a=\alpha^2=4\beta^2\not\equiv 1$ mod 8.   
Thus, we have $\alpha^2-a=(2\beta+1)^2-(8b+1)=4(\beta^2+\beta-2b)$. Since the polynomial $Y^2+Y-2b$ mod 2 has single roots in $\F_2$, Hensel's lemma yields roots of $Y^2+Y-2b$ in $\Z_2$, and hence those of $X^2-a$ in $\Z_2$. 

ii) Suppose $p\neq 2$. If $p\mid a$, then, by the assumption $p^2\nmid a$, we have
$a=pu$ with $u\in \Z_p^\times$, and hence $\sqrt{a}\not\in \Zp$. 
If $p\nmid a$, then, since $p\neq 2$, $X^2-a$ mod $p$ has single roots in $\F_{p^2}$ and they lift to a quadratic extension of $\Zp$ by Hensel's lemma. 
In addition, the following are equivalent; $(\frac{\,a\,}{p})=1$, the roots of $X^2-a$ mod $p$ belong to $\F_p$, the roots of $X^2-a$ belong to $\Zp$. 
This completes the proof.  
\end{proof}

\begin{lem} \label{lem.liminal}
There is a bijective correspondence between liminal ${\rm SL}_2\Zp$-characters of $\pi_1(S^3-J(2k,2l))$ and the intersection points of $f_{k,l}(x,y)=0$ and $y-2=0$ in $\Zp^{\,2}$. 
\end{lem}

\begin{proof} 
Recall that the varieties of reducible and irreducible characters are given by $y-2=0$ and $f_{k,l}(x,y)=0$.  
If $\rho$ is a liminal representation with ${\rm Im}\,{\rm tr}\,\rho\subset \Zp$, then ${\rm tr}\,\rho$ is on $y-2=0$ in $\Zp^{\,2}$. 
In addition, for every $n\in \Z_{>0}$, there is an irreducible representation $\rho'$ with $\rho\equiv \rho'$ mod $p^n$, so ${\rm tr}\,\rho$ is on $f_{k,l}(x,y)\equiv 0$ mod $p^n$. Thus ${\rm tr}\,\rho$ is on $f_{k,l}(x,y)=0$ in $\Zp^{\,2}$ as well. 

Let us prove the converse. 
Note that Riley's representation $\rho^{\rm R}$ in \Cref{sec.DTK} may be seen as a representation over a quadratic extension of $\Zp[x,y]$, via $s^2-xs+1=0$. 

By \Cref{prop.Jredirr}, 
if $f_{k,l}(x,y)=0$ and $y-2=0$ has an intersection point in $\Z_p^{\,2}$, then 
$f_{k,l}(x,2)=(f_{k,l}(x,y)\ {\rm mod}\ (y-2))$ has a single root $\alpha$ in $\Zp$, 
and Hensel's lemma yields the implicit function $x=x_f(y) \in \Zp[\![y-2]\!]$ around $(x,y)=(\alpha,2)$. 
By substituting $x=x_f(y)$ in Riley's representation, 
we obtain an irreducible representation $\bs{\rho}^R$ over a quadratic extension of $\Zp[\![y-2]\!]$ with 
${\rm Im}\,{\rm tr}\,\bs{\rho}^R \subset \Zp[\![y-2]\!]$. 

Note that every element of $\Zp[\![y-2]\!]$ converges in $\Z_p$ at any $y\in \Z_p$ in the $p$-adic unit disc $|y-2|_p<1$.   
If we substitute $y=2$, then this $\bs{\rho}^R$ yields a reducible representation $\rho$ with ${\rm tr}\,\rho\subset \Zp$ 
at $(\alpha,2)$. 
If we instead substitute $y=\eta \in \Zp$ with $0\neq |2-\eta|_p<1$, namely, $2\equiv \eta$ mod $p^n$ for some $n\in \Z_{>0}$, then $\bs{\rho}^R$ yields an absolutely irreducible representation $\rho_\eta$ with ${\rm tr}\,\rho_\eta\subset \Zp$. 
Thus, ${\rm tr}\,\rho$ is a liminal ${\rm SL}_2\Zp$-character. 
\end{proof} 

\begin{proof}[Proof of \Cref{thm.limchar}] 
Since  $f_{k,l}(x,y)=0$ and $y-2=0$ have an intersection point in $\Zp^{\,2}$ if and only if $p\nmid kl$ and $\sqrt{4k^2l^2-kl}\in \Zp$, \Cref{lem.Hensel} and \Cref{lem.liminal} yield the assertion. 
\end{proof}

\begin{eg} \label{eg.J(22l)}
The condition for $\pi_1(S^3-J(2,2l))$ admitting a liminal ${\rm SL}_2\Zp$-character becomes 
i) $p=2$ and $4l^2-l\equiv 1$ mod 8 
or ii) $p\neq 2$ and $(\frac{4l^2-l}{p})=1$. 
By elementary calculation, we obtain the following. 

(i) $J(2,2)=3_1$ (trefoil): $(\frac{\,3\,}{p})=1$, i.e., $p\equiv \pm 1$ mod 12. 

(ii) $J(2,-2)=4_1$ (figure-eight): $(\frac{\,5\,}{p})=1$, i.e., $p\equiv \pm1$ mod 5. 

(iii) $J(2,4)$: $(\frac{14}{p})=1$, i.e., 
$p\equiv \pm1,\pm9,\pm 25, \pm 5, \pm 11, \pm 13$ mod 56. 

(iv) $J(2,-4)$: $(\frac{18}{p})=1$, i.e., $p\equiv \pm1$ mod 8. 

(v) $J(2,6)$: $p=2$ or $(\frac{33}{p})=1$, i.e., 
$p\equiv \pm1,\pm2,\pm4,\pm8,\pm16$ mod 33. 

(vi) $J(2,-6)$: $(\frac{39}{p})=1$, i.e., $\pm p %\equiv 1,25,133,49,121,61, 41, 89, 5, 125, 149, 137
\equiv 1,5,7,19,23,25,26, 35,41,49,61,67$ 
%119,11,83,47,71,59$ 
 mod 156. 
\end{eg}

\section{Liminal ${\rm SL}_2\Zp$-representations} 
Let $\rho$ be a representation over a finite extension of $\Zp$
such that $\ol{\rho}:=\rho$ mod $p$ is absolutely irreducible and suppose that ${\rm Im}\,{\rm tr}\,\ol{\rho}\subset \F_p$. 
Since the Brauer group of a finite field is trivial, the Skolem--Noether theorem assures that 
such $\ol{\rho}$ is conjugate to some representation $\ol{\rho}'$ over $\F_p$ (cf.\cite[Subsection 3.5]{Marche-RIMS2016}).
In addition, Nyssen's theorem \cite[Th\'eor\`em 1]{Nyssen1996} asserts that $\ol{\rho}'$ lifts to some $\rho'$ over $\Zp$ such that ${\rm tr}\,\rho={\rm tr}\,\rho'$, and such $\rho'$ is unique up to conjugate. 

One might wonder whether this argument extends to representations on the Zariski closure of the variety $X_{\rm irr}(S^3-J(2k,2l))=\{f_{f,l}(x,y)=0\}\setminus \{y-2=0\}$ of irreducible characters. 
Nyssen's proof uses the fact that being absolutely irreducible is equivalent to the surjectivity of the corresponding algebra homomorphism, so the result does not necessarily extend to residually reducible representations. 
In addition, as pointed out in \cite[Subsection 2.3]{KMTT2018}, in the setting of knot group representations, the deformation problem is not unobstructed in the sense that the 2nd cohomology of the adjoint representation does not vanish. 
The following question seems subtle. 
\begin{q} \label{q.limrep} 
Given a liminal character over $\Zp$, 
can we find a liminal ${\rm SL}_2\Zp$-representation? 
\end{q} 

We may say at least the following. 

\begin{prop} \label{prop.limZp} 
Under the conditions of \Cref{thm.limchar}, if 
$p\neq 2$ and $(\frac{-kl}{p})=1$, 
then the group of $J(2k,2l)$ admits 
%there exists 
a liminal ${\rm SL}_2\Zp$-representaion. 
\end{prop} 
\begin{proof} 
Let $\bs{\rho}^R$ be the one in \Cref{lem.liminal} and note that 
$s=\frac{x-\sqrt{x^2-4}}{2}$. 
By \Cref{prop.Jredirr}, we have $\sqrt{x^2-4}\,|_{y=2}=\sqrt{\frac{-1}{kl}}$. 
If $(\frac{-kl}{p})=1$, then by \Cref{lem.Hensel} (ii), $\sqrt{x^2-4}=\sqrt{\frac{-1}{kl}}\in \Zp$, and 
$\bs{\rho}^R$ mod $y-2$ becomes a liminal ${\rm SL}_2\Zp$-representation. 
\end{proof} 

\begin{rem} \label{rem.2bridge} 
We remark that most argument above persists for any two-bridge knot $K$ with the ${\rm SL}_2$-character variety $f(x,y)(y-2)=0$, yielding that,  
if $f(x,2)$ has a single root in $\F_p$, then $K$ admits a liminal ${\rm SL}_2\Zp$-character. 
The existence of a liminal ${\rm SL}_2\Zp$-representation remains in mystery. 
\end{rem}

\section{Lucas and Fibonacci type sequences}
Let $m\in \Z$ and write $t^2-t+m=(t-a)(t-b)$, so that we have $a+b=1$ and $ab=m$. 
We define the Lucas-type sequence and the Fibonacci-type sequence 
by $L_n=a^n+b^n$ and $F_n=\frac{a^n-b^n}{a-b}$ respectively. 
Then, we have $L_0=2$, $L_1=1$, $L_2=1-2m$, $L_{n+2}=L_{n+1}-mL_n$, 
$F_0=0$, $F_1=1$, $F_2=1$, $F_{n+2}=F_{n+1}-mF_n$, and hence $L_n,F_n\in \Z$. 
We have 
\[L_n^{\,2}+(4m-1)F_n^{\,2}=4m^n.\ \ \cdots (\star)\]
Indeed, by $(a-b)^2=(a+b)^2-4ab=1-4m$, we have 
$(a^n+b^n)^2+(4m-1)(\frac{a^n-b^n}{a-b})^2
=(a^n+b^n)^2-(a^n-b^n)^2
=4a^nb^n=4m^n$. 

Now suppose that $p\neq 2$ divides $L_{2n+1}$ for some $n\in \Z_{\geq 0}$. Then we have $(4m-1)F_{2n+1}^{\,2}\equiv (2m^{n})^2m$ mod $p$. 
This implies that $m(4m-1)$ must be a square mod $p$. 
In addition, we have $p\nmid m(4m-1)$. Indeed, if $p\mid m$, then the recurrence formula yields that $L_n\equiv 1$ mod $p$ for all $n\in \Z_{>0}$, contradicting $p\mid L_{2n+1}$. If instead $p\mid 4m-1$, then $(\star)$ yields $p=2$ or $p\mid m$, hence again contradiction. 
%Thus, we have the following. 

Together with a direct calculation, we obtain the following.  
\begin{thm} \label{thm.Lucas}
{\rm (i)} If $2^3$ divides $L_{2n+1}$ for some $n\in \Z_{\geq 0}$, then $4m^2-m\equiv 1$ mod 8 holds. 

{\rm (ii)} If an odd prime number $p$ divides $L_{2n+1}$ for some $n\in \Z_{\geq 0}$, then the Legendre symbol satisfies $\left(\dfrac{4m^2-m}{p}\right)=1$. 
\end{thm} 

\begin{proof} The assertion (ii) is done. To verify (i), note that $(L_n\, {\rm mod}\, 8)_n$ with $n\geq 0$ is periodic for $n\gg 0$. In mod 8, we have:  
\begin{center}
\begin{tabular}{|c||l|} \hline 
$m$ mod 8 & $L_n$ mod 8\\ \hline 
0&2,1,1,1,...\\
1& $2, 1, 7, 6, 7, 1, 2, 1$, ...\\
$-1$ & 2,1,3,4,7,3,2,5,7,4,3,7,2,1, ... \\ 
2& 2,1,5,3,1,3,1, ...  \\
$-2$& 2,1,5,7,1,7,1, ... \\ 
3& 2,1,3,0,7,7,2,5,7,0,3,3,2,1, ... \\ 
$-3$& 2,1,7,2,7,5,2,1, ... \\ 
4& 2,1,1,5,1,5, ... \\ \hline 
\end{tabular}
\end{center} 
%
\begin{comment} 
If $m\equiv 0$, then $L_n\equiv 2,1,1,1,$ ... .

If $m\equiv 1$, then $L_n\equiv 2, 1, -1, -2, -1, 1, 2, 1$, ... .

If $m\equiv -1$, then $L_n\equiv 2,1,3,4,7,3,2,5,7,4,3,7,2,1,$ ... . 

If $m\equiv 2$, then $L_n\equiv 2,1,5,3,1,3,1,$ ... . 

If $m\equiv -2$, then $L_n\equiv 2,1,5,7,1,7,1,$ ... . 

If $m\equiv 3$, then $L_n\equiv 2,1,3,0,7,7,2,5,7,0,3,3,2,1,$ ... .

If $m\equiv -3$, then $L_n\equiv 2,1,7,2,7,5,2,1,$ ... .

If $m\equiv 4$, then $L_n\equiv 2,1,1,5,1,5,$ ... . 
\end{comment} 
%
Thus, we have $2^3\mid L_{2n+1}$ for some $n\in \Z_{\geq 0}$ if and only if $m\equiv 3$ mod 8 holds. 
%Note that we have $L_n=0$ for some odd $n$ iff $m=3$. 
Since the condition $4m^2-m\equiv 1$ mod 8 is equivalent to that $m\equiv 3$ mod 8, we obtain the assertion. 
\end{proof}

\begin{eg} The following may be compared with \Cref{eg.J(22l)}. 

%(i) 
Let $m=1$. Then the condition $(\frac{12}{p})=1$ becomes 
$p\equiv \pm1$ mod 12. 
By $L_{n+1}=L_n-L_{n-1}$, we see that $L_n$ is periodic as 
\begin{center}
\begin{tabular}{c||c|c|c|c|c|c}
$n$ mod 6&1&2&3&4&5&6\\ \hline 
$L_n$&1&$-1$&$-2$&$-1$&1&2
\end{tabular},
\end{center} 
so only $p=2$ actually appears. 

%(ii) 
Let $m=-1$. Then, $L_n$ and $F_n$ become the classical ones, and the condition 
$(\frac{\,5\,}{p})=1$ becomes %$p=2$ or 
$p\equiv \pm1$ mod 5. We actually have 
\begin{center}
\begin{tabular}{c||c|c|c|c|c|c|c|c|c|c|c|c|c|c}
$n$&1&2&3&4&5&6&7&8&9&10&11&12&13&14 \\ \hline 
$L_n$&1&3& $2^2$ &7&11& $2\ccdot 3^2$ &29&47& $2^2\ccdot 19$ & $3\ccdot 41$ &199& $2\ccdot 7\ccdot 23$ & 521 & $3\ccdot 281$ 
\end{tabular}

%\hspace{1cm} 
\begin{tabular}{c|c|c|c|c|c|c|c}
%$n$
\,&15&16&17&18&19&20&21\\ \hline 
%$L_n$
\,& $2^2 \ccdot 11 \ccdot 31$ & 2207 & 3571 & $2\ccdot3^2\ccdot321$ & 9349 & $127\ccdot119$& $2^2\ccdot 6119$ 
\end{tabular}. 
\end{center}

We list the value of $L_{2n+1}$ with $0\leq n\leq 10$ for 
$m=-1,2,-2,3,-3,4$: 
\begin{center}
{\scriptsize
\begin{tabular}{c||c|c|c|c|c|c|c|c|c|c|c}
$m\, \backslash\,2n+1\!\!$&1&3&5&7&9&11&13&15&17&19&21\\ \hline 
$-1$ &1 &$2^2$ &11 &29 & $2^2\ccdot 19$ &199 & 521 & $2^2 \ccdot 11 \ccdot 31$ & 3571 & 9349 & 
$2^2\ccdot 6119$\\ 
2&1&$-5$&11&$-13$&$-5$&67&$-181$&$5^2\ccdot 11$ &$-101$&$-797$ & $5\ccdot 13\ccdot 43$\\
$-2$&1&7&31&127&$7\ccdot 73$&$23\ccdot 89$&8191&$7\ccdot 31\ccdot 151$&131071&524287&
$7\ccdot 299593$\\ 
3&1&$-2^3$&31&$-83$&$2^3\ccdot17$&67&$-1559$&$2^3\ccdot29\ccdot31$&$-21929$&44917&$-2^3\ccdot41\ccdot83$\\ 
$-3$&1&$2\ccdot5$&61&337&$2\ccdot5\ccdot181$&$23\ccdot491$&51169&$2\ccdot5^2\ccdot 5429$
& $1439629$
& $7634353$
&$2\ccdot 5\ccdot 4048381$\\ 
4&1&$-11$&61&$-251$&$11\ccdot 71$&$-1451$&$-2339$&$11\ccdot 59\ccdot 61$&$-239699$&$229\ccdot 451$&$-11\ccdot251\ccdot1259$\\
\end{tabular} 
}
\end{center}
We may often find a large prime number there. 
\end{eg} 

\begin{rem} At least when $m=\pm 1$, the converse of \Cref{thm.Lucas} %nor \Cref{thm} 
does not hold. 
When $m=-1$ (the classical case), we can say more:  
%Indeed, 
Lagarias \cite{Lagarias1985PJM} proved that the density of 
\[S_A:=\{p\mid p\equiv \pm1\ \text{mod}\ 5\ \text{and}\ p\ \text{divides}\ L_n\ \text{for\ some}\ n\in \Z_{>0}\}\] in the set of prime numbers is 5/12. 
So, the density of 
\[\{p\mid p\equiv \pm1\ \text{mod}\ 5\ \text{and}\ p\ \text{divides}\ L_{2n+1}\ \text{for\ some}\ n\in \Z_{\geq0}\}\] 
is smaller than the density 1/2 of the set $\{p\mid p\equiv \pm1\ \text{mod}\ 5\}$. 
\end{rem}

\section{Cyclic covers} 
Let $\Delta_K(t)$ denote the Alexander polynomial of a knot $K$ in $S^3$. 
Then Fox--Weber's formula \cite{Weber1979} asserts that the $\Z/n\Z$-cover $M_n\to S^3$ branched over $K$ satisfies 
\[r_n:=|H_1(M_n;\Z)|=|{\rm Res}(t^n-1, \Delta_K(t))|\] %=|\prod_{\zeta^n=1}\Delta_K(\zeta)|$
for each $n\in \Z_{>0}$. 
Here, if $G$ is a finite group, then $|G|$ denotes the order $\#G$ of $G$, and if $G$ is an infinite group, then we put $|G|=0$. 
For polynomials $f(t), g(t)\in \Z[t]$, ${\rm Res}(f(t),g(t))\in \Z$ denotes their resultant.
We remark that if $K$ is a two-bridge knot of genus one, then the precise group structure of $H_1(M_n)$ is known by Fox and others (\cite{Fox1960AnnMath}, see also \cite{Mednykh2021arXiv}). 
Now, we assert the following. 
\begin{thm} \label{thm.pmidrn} 
Suppose $\Delta_K(t)=mt^2+(1-2m)t+m$ with $m\in \Z$. 

{\rm (i)} If $2^3\mid r_n$ for some odd $n$, then $4m^2-m\equiv 1$ mod 8 holds. 

{\rm (ii)} If $p\neq 2$ and $p\mid r_n$ for some odd $n$, then $\left(\dfrac{4m^2-m}{p}\right)=1$ holds. 
\end{thm} 

\begin{proof} 
If $m=0$, then $r_n=1$ and $p\nmid r_n$. 
Suppose $m\neq 0$. 
Let us write 
$\Delta_K(t)=mt^2+(1-2m)t+m=m(t-\alpha)(t-\beta)$. 
Then, we have $\alpha\beta=1$, $\alpha+\beta=\frac{2m-1}{m}$, and 
${\rm Res}(\Delta_K(t),t^n-1)$ 
$=m^n(\alpha^n-1)(\beta^n-1)
=m^n(2-\alpha^n-\beta^n)$.
In addition, let us write $t^2-t+m=(t-a)(t-b)$, so that we have $a+b=1$, $ab=m$, 
$a^2+b^2=1-2m=-m(\alpha+\beta)$,
$a^2b^2=m^2=(-m\alpha)(-m\beta)$, and hence 
$\{a^2,b^2\}=\{-m\alpha,-m\beta\}$. 
We may assume $a^2=-m\alpha$, $b^2=-m\beta$. 
If $n$ is odd, then we have 
$L_n^{\,2}=(a^n+b^n)^2
=((-m\alpha)^n+(-m\beta)^n+2m^n)
=m^n(2-\alpha^n-\beta^n)$
$={\rm Res}(\Delta_K(t),t^n-1)=-{\rm Res}(t^n-1,\Delta_K(t))$ 
$= r_n$. 
Hence, the assertion follows from \Cref{thm.Lucas}. 
\end{proof} 

\begin{lem} \label{lem.finite} 
Let $\Delta_K(t)=mt^2+(1-2m)t+m$ with $m \in \Z$. 
Then, $H_1(M_n)$ is an infinite group if and only if $(m,n)=(1,6)$. 
\end{lem} 
\begin{proof} 
For each $n\in \Z_{>0}$, let $\Phi_n(t)\in \Z[t]$ denote the $n$-th cyclotomic polynomial over $\Q$, so we have $[\Q(\zeta_n):\Q]={\rm deg}\,\Phi_n(t)=\varphi(n)$, where $\varphi(n)$ denotes Euler's totient function. 
Since $[\Q(\zeta_n):\Q]=2[\Q(\cos \frac{2\pi}{n}):\Q]$, we have $\varphi(n)\leq 2$ if and only if $n\in \{1,2,3,4,6\}$. 

If $H_1(M_n)$ is an infinite group, then $r_n=|{\rm Res}(t^n-1,\Delta_K(t))|=0$, so we have $\Phi_n(t)\mid \Delta_K(t)$ and $m\neq 0$. 
By $\Delta_K(t)/m=t^2-(2-\frac{1}{m})t+1=(t-\zeta_n)(t-\zeta_n^{-1})$ with $\zeta_n=e^{2\pi\sqrt{-1}/n}$, 
we must have $\Delta_K(t)/m\in\{(t+1)^2,(t-1)^2,t^2+t+1,t^2+1,t^2-t+1\}$. 
By $2-\frac{1}{m}=2\cos\frac{2\pi}{n} \in \Z$, we have $m=1$ and $n=6$, namely, $\Delta_K(t)=t^2-t+1=\Phi_6(t)$. 
%We have $r_n=0$ only if $m=1$ and $n=6$, so $r_n$ coincides with the size of the group $H_1(M_n)$ for every odd $n$.
\end{proof}

\begin{proof}[Proof of \Cref{thm}] 
Since 
$\Delta_{J(2k,2l)}(t)=klt^2+(1-2kl)t+kl$, 
by putting $m=kl$, Theorems \ref{thm.limchar} and \ref{thm.pmidrn}, 
together with \Cref{lem.finite}, yield the assertion. 
\end{proof}

\begin{rem}
If $n$ is even, then we have 
${\rm Res}(t^n-1,\Delta_k(t))=m^n(\alpha^n+\beta^n-2)=((-m\alpha)^n+(-m\beta)^n-2m^n)=(a^n-b^n)^2=(a-b)^2F_n^{\,2}=(1-4m)F_n^{\,2}=L_n^{\,2}-(2m^{n/2})^2$. 
Thus, $p\mid r_n$ does not lead to a similar condition of the Legendre symbol. 
\end{rem}

\begin{eg} 
If $4m^2-m$ has a square factor, then the converse of \Cref{thm} does not hold. 
For instance, if $m=16$ and $p=3$, then $4m^2-m=16\ccdot 63=7\cdot 12^2$, $(7/3)=(1/3)=1$. 
So $J(2,32)$ admits a liminal character over $\Z_3$. 
In this case $3\nmid r_{2n-1}$ for all $n\in \Z_{>0}$. 
\end{eg}  

\begin{rem} \label{rem.final}
The following are still a mystery. 

(1) Does \Cref{thm} extend to some wider class of knots? 
We may further verify by using Mathematica \cite{Mathematica2025} and SageMath \cite{sage} that a similar assertion holds for genus two two-bridge knots 
\begin{itemize}
\item[$\bullet$] $K=6_2$ with $n< 500$ and $p<10000$ except for $p=2,7,53,$ 
\item[$\bullet$]  $K=6_3$ with $n<500$ and $p<2000$ except for $p=2,53,101.$
\end{itemize} 
%\blue{$K=6_2$ with $n< 500$ and $2<p<10000$ and for $6_3$ with $n<500$ and $2<p<2000$ 
%except for $(K,p)=(6_2,7), (6_2,53), (6_3,53), (6_3,101)$.} 
For instance, %For $p=2$, 
we find that $K=6_{2}$ with $\Delta_K(t)=-t^4+3t^3-3t^2+3t-1$ satisfies $r_5=|{\rm Res}\,(t^5-1,\Delta_K(t))|=2^4$, while the intersection of two varieties is empty in $\F_2^{\,2}$. 
%Also, we have $r_{25}=-1 * 2^4 * 7^4 * 149^2$. 
We expect that there is a correct formulation or interpretation in a more general setting. 
% and SageMath that a similar 
%It turned out that for $K=6_2$, $p=7$ is also a counter-example. 
%$r_{25}$ is divisible by $7$ but the intersection of two varieties is empty in $\F_7^{\,2}$. 

(2) 
One may believe that the Burde--de Rham theory is in the background of \Cref{thm}. 
However, there is no obvious implication, and we rather expect that our result will shed new light around there. 

(3) The Burde--de Rham theory %, which is in the background of \Cref{thm}, 
has been extended to higher-dimensional representations by Heusner--Porti \cite{HeusenerPorti2015PJM}. 
We wonder if our study also extends to higher-dimensional cases. 
The study of ${\rm SL}_3\C$-character variety of the figure-eight knot \cite{HeusenerMunozPorti2016Illinois} might be a clue. 

(4) As we asked in \Cref{q.limrep}, it seems to be a subtle question to ask 
%we wonder 
if there always exists a liminal ${\rm SL}_2\Zp$-representation under the condition of \Cref{thm.limchar}. 

(5) Prof. Ohshita recently pointed out that from a viewpoint of number theory, our work seems like a partial analogue of the converse of the Herbrand--Ribet theorem, or an analogue of Sharifi's conjecture (cf.\cite{FuyakaKato2024Kyoto}) that suggests a relation between the divisibility of the class numbers by $p$ and a certain congruence of certain cusp forms.
\end{rem} 

%\cite{SakamotoTangeUeki2025-arXiv} \cite{SakamotoTangeUeki-knot2024} \cite{Morishita2012} \cite{Morishita2024} 

\bibliographystyle{amsalpha} %amsplain %plain}% jabbrv} 
\bibliography{%uekitateno.ju.arXiv.bbl}
%refs1}
SakamotoTangeUeki.Liminal.arXiv.bbl} 
%/Users/uekijun/Dropbox/refs1} 

\begin{comment}
%{\scriptsize 
\section*{Statement and Declarations} 
Competing Interests: On behalf of all authors, the corresponding author states that there is no conflict of interest. 

Data availability statements: %Data sharing not applicable to this article as no datasets were generated or analysed during the current study.
This manuscript has no associated data. 
%}
\end{comment} 

\ 

\end{document}